\documentclass[a4paper,12pt]{article}
\usepackage{amsfonts}
\usepackage{pifont}
\usepackage{bm}
\usepackage{latexsym,amsmath,amssymb,cite,amsthm}
\usepackage{color,eucal,enumerate,mathrsfs}
\usepackage[normalem]{ulem}
\usepackage{amsmath}
\usepackage[colorlinks, linkcolor=blue,  anchorcolor=blue,citecolor=blue]{hyperref}

\numberwithin{equation}{section}
\theoremstyle{plain}
\newtheorem{theorem}{Theorem}[section]
\newtheorem{corollary}[theorem]{Corollary}

\newtheorem{example}[theorem]{Example}
\newtheorem{conjecture}[theorem]{Conjecture}

\theoremstyle{definition}
\newtheorem{definition}[theorem]{Definition}
\theoremstyle{remark}

\newcommand{\lmt}[2]{\mathop{\lim}_{{#1} \rightarrow {#2}} }

\newcommand{\mm}{\mathfrak m}
\newcommand{\ms}{(X,\d,\mm)}
\newcommand{\cd}{{\rm CD}(0, \infty)}
\newcommand{\cdkn}{{\rm CD}(0, N)}
\newcommand{\mcpkn}{{\rm MCP}(0, N)}

\newcommand{\rcd}{{\rm RCD}(0, \infty)}

\newcommand{\ent}[1]{{\rm Ent}_{#1}}
\newcommand{\R}{\mathbb{R}}

\newcommand{\diam}{\mathop{\rm diam}\nolimits} 

\newcommand{\supp}{\mathop{\rm supp}\nolimits}   %%\newcommand{\span}{\mathop{\rm span}\nolimits}   %
\renewcommand{\d}{{\mathrm d}}

\newcommand{\restr}[1]{\lower3pt\hbox{$|_{#1}$}}

\newcommand{\nchi}{{\raise.3ex\hbox{$\chi$}}}

\title{\large{\bf A sharp isoperimetric inequality in metric measure spaces with non-negative Ricci curvature}
}

\begin{document}
\author{Bang-Xian Han\thanks{Wu Wen-Tsun Key Laboratory of Mathematics and School of   Mathematical Sciences, University of Science and Technology of China (USTC), 230026, Hefei, China.  Email: hanbangxian@ustc.edu.cn} }

\date{\today} 
\maketitle

%%%%%%%%%%%%%%
% 摘要还需要更加详细
%%%%%%%%%%%%%%
\begin{abstract}
We prove a sharp dimension-free isoperimetric inequality,  involving the volume entropy,  in non-compact metric measure spaces with non-negative synthetic Ricci curvature.  
\end{abstract}

\tableofcontents

\section{Introduction}
The aim  of this  note is to present a sharp isoperimetric inequality for a class of metric measure spaces with  non-negative Ricci curvature in the sense of Lott-Sturm-Villani.

Let $\ms$ be a metric measure space. Recall that the Minkowski content of a  Borel
set $Z \subset X$ with $\mm(Z)<+\infty$ is defined by
\[
\mm^+(Z):=\mathop{\liminf}_{\epsilon \to 0} \frac{\mm(Z^\epsilon)-\mm(Z)}{\epsilon}
\]
where $Z^\epsilon \subset X$ is the $\epsilon$-neighbourhood of $Z$ defined by $Z^\epsilon:=\{x: \d(x, Z)<\epsilon\}$. Isoperimetric inequalities for a class $\mathcal M$ of metric measure spaces  relate the size of the boundary of sets to their measure. Precisely, there is  a function $I_{\mathcal M}(v)$, called isoperimetric profile, such that 
\[
\mm^+(\Omega) \geq I_{\mathcal M}(v)
\]
for all $\ms \in \mathcal M$ and $\Omega \subset X$ with $\mm(\Omega)=v$.

Recently,   isoperimetric inequalities in the class of  non-compact metric measure spaces with synthetic non-negative Ricci curvature  is studied in \cite{balogh2021sharp} and \cite{cavalletti2021isoperimetric}.  In these isoperimetric inequalities,  a key number in the isoperimetric profile is  a  constant called  {\sl asymptotic volume ratio} of a metric measure space $\ms$, defined as 
\begin{equation*}
{\rm AVR}\ms:=\lmt{r}{+\infty} \frac {\mm\big (B_r(x_0)\big)} {r^N}\in [0, +\infty].
\end{equation*}
When ${\rm AVR}\ms>0$, we call that the space $\ms$ has Euclidean volume growth.

Since the dimension upper bound $N$ appear in the asymptotic volume ratio, the  isoperimetric inequalities  obtained in  \cite{balogh2021sharp} and \cite{cavalletti2021isoperimetric} are all dimension-dependent. So it is natural to find a dimension-free version of the isoperimetric inequality for metric measure spaces with  non-negative Ricci curvature. Firstly, recall that a space has non-negative Ricci curvature   (without dimension restriction) in the sense of Lott-Sturm-Villani means:

\begin{definition}[Lott-Sturm-Villani's curvature-dimension condition, cf. \cite{Lott-Villani09, S-O1}]\label{def:cd}
 We say that a metric measure space $\ms$  has non-negative Ricci curvature, or satisfies $\cd$ condition,  if  the entropy  functional  $\ent{\mm}$ is  displacement convex  on  the $L^2$-Wasserstein space $(\mathcal{P}_2(X), W_2)$. This means,  for any two probability measures $\mu_0, \mu_1 \in \mathcal {P}_2 (X)$ with $\mu_0, \mu_1 \ll \mm$, there  is  a $L^2$-Wasserstein geodesic $(\mu_t)_{t\in [0,1]}$ such that 
 \begin{equation}\label{eq1.5-intro}
{\rm Ent}_\mm(\mu_t) \leq t{\rm Ent}_\mm(\mu_1)+(1-t){\rm Ent}_\mm(\mu_0)
\end{equation}
where ${\rm Ent}_\mm(\mu_t)$ is defined as $\int \rho_t \ln \rho_t\,\d \mm$ if $\mu_t=\rho_t\,  \mm$, otherwise  ${\rm Ent}_\mm(\mu_t)=+\infty$.
\end{definition}

In the study of metric (Riemannian) geometry, there are some important spaces where the asymptotic volume ratio ${\rm AVR}=+\infty$. In these cases, we often consider instead the {\sl volume entropy}, which is an important concept in both Riemannian  geometry (cf.  \cite{BessonEntropy})  and dynamical system (cf. \cite{ManningEntropy} ).
\begin{definition}[Volume entropy]\label{def:ve}
A metric measure space $\ms$ admits the {\sl volume entropy} at $x_0\in X$, denoted by $h\ms(x_0)$,  provided
\[
h\ms(x_0):=\lmt{r}{+\infty} \frac {\ln \mm\big (B_r(x_0)\big)} {r}\in [0,\infty].
\]
\end{definition}

It is not hard to prove that a metric measure space with non-negative Ricci curvature in the sense of Definition \ref{def:cd}, surely admits the volume entropy which is independent of the choice of $x_0$

The main result of this note is the following sharp  isoperimetric inequality involving volume entropy:

\begin{theorem}[Sharp isoperimetric inequality]\label{th1}
Let $\ms$ be a   metric measure space with non-negative synthetic Ricci curvature.  Then for any $\Omega \subset X$ with finite measure, it holds the following isoperimetric inequality
\begin{equation}\label{intro:eq1}
\mm^+(\Omega) \geq \mm(\Omega) h\ms.
\end{equation}
Moreover, the constant $h\ms$ in \eqref{intro:eq1} can not be replaced by any larger ones.
\end{theorem}

\section{Main results}\label{sect:main}
In the following two theorems we will prove a sharp isoperimetric inequality for the class of metric measure spaces with  non-negative  synthetic Ricci curvature.
\begin{theorem}[Dimension-free isoperimetric inequality]\label{th1}
Let $\ms$ be a   metric measure space with non-negative Ricci curvature.  Then for any $\Omega \subset X$ with finite measure, it holds the following isoperimetric inequality
\begin{equation}\label{th1:eq1}
\mm^+(\Omega) \geq \mm(\Omega) h\ms.
\end{equation}
\end{theorem}
\begin{proof}
Given $x_0 \in X$ and  $R>0$. 
Define $\mu_0=\frac 1{\mm(\Omega)} \mm \restr{\Omega}$ and $\mu_1=\frac 1{\mm(B_R(x_0))} \mm \restr{B_R(x_0)}$.  By \ref{def:cd} there is an  $L^2$-Wasserstein geodesic  $(\mu_t)$ connecting $\mu_0, \mu_1$ such that
 \begin{equation}\label{th1:eq2}
{\rm Ent}_\mm(\mu_t) \leq t{\rm Ent}_\mm(\mu_1)+(1-t){\rm Ent}_\mm(\mu_0)
\end{equation}
Denote the support of $\mu_t$ by $\supp \mu_t$ and the set of $t$-intermediate points $$Z_t:=\Big \{z: \exists~ x\in \Omega, y\in B_R(x_0), \text{such that}~ \frac {\d(z, x)}t=\frac{\d(z, y)}{1-t}=\d(x, y) \Big\}.$$
It can be seen that $ \supp \mu_t  \subset Z_t$. Then by \eqref{th1:eq2} and Jensen's inequality, we have
 \begin{equation}\label{th1:eq3}
-\ln \big (\mm(Z_t)\big) \leq  -\ln \big (\mm(\supp \mu_t)\big)  \leq -t \ln \big (\mm(B_R(x_0))\big ) - (1-t) \ln \big(\mm(\Omega) \big).
\end{equation}

Let $\epsilon:=t(\diam (\Omega)+R)$.  It is not hard to see that $Z_t\subset \Omega^\epsilon$.   If $\mm^+(\Omega)=+\infty$, there is nothing to prove. Otherwise,  we have
$\lmt{\epsilon}0\mm(\Omega^\epsilon)=\mm(\Omega)$.  Then we have
\begin{eqnarray*}
\frac{\mm^+(\Omega) }{\mm(\Omega)}&=& \mathop{\liminf}_{\epsilon \to 0}  \frac 1{\mm(\Omega) } \frac{\mm(\Omega^\epsilon)-\mm(\Omega)}{ \epsilon}\\
&=&   \mathop{\liminf}_{\epsilon \to 0} \frac{\ln \big(\mm(\Omega^\epsilon)\big)- \ln \big(\mm(\Omega)\big)}  {\mm(\Omega^\epsilon)-\mm(\Omega)}  \frac {\mm(\Omega^\epsilon)-\mm(\Omega)}{ \epsilon}\\
&\geq &    \mathop{\liminf}_{t \to 0} \frac{\ln\big(\mm(Z_t)\big)-\ln \big(\mm(\Omega) \big)}{t(\diam (\Omega)+R)}\\
\text {By}~ \eqref{th1:eq3}&\geq&  \mathop{\liminf}_{t \to 0} \frac{t \ln \big (\mm(B_R(x_0))\big ) +(1-t) \ln \big(\mm(\Omega) \big)-\ln\big(\mm(\Omega)\big)}{t(\diam (\Omega)+R)}\\
&=&   \frac{ \ln \big (\mm(B_R(x_0))\big )-\ln\big(  \mm(\Omega)\big)}{\diam (\Omega)+R}.
\end{eqnarray*}
Letting $R \to \infty$,  by  Definition \ref{def:ve}, we get
\[
\frac{\mm^+(\Omega) }{\mm(\Omega)} \geq h\ms
\]
which is the thesis.
\end{proof}
\bigskip
\begin{theorem}[Sharp inequality]\label{th2}
The inequality \eqref{th1:eq1} in Theorem \ref{th1} is sharp. This means, for any $C> h\ms$, the inequality ${\mm^+(\Omega) } \geq C {\mm(\Omega)}$ does not hold for any Borel set  $\Omega \subset X$.
\end{theorem}
\begin{proof}
We will prove the theorem by contradiction. Assume there is a constant $C> h\ms$, such that 
 \begin{equation}\label{th2:eq1}
{\mm^+(\Omega) } \geq C {\mm(\Omega)}
\end{equation}
  for any Borel set $\Omega \subset X$.

Apply \eqref{th1:eq3} with  $\Omega=B_{r+\delta}(x_0)$  and $R=\epsilon$.   We get the following interpolation inequality 
 \begin{equation}\label{prop1:eq1}
\ln \big (\mm(Z_t)\big) \geq  t \ln \big (\mm(B_\epsilon (x_0))\big ) +(1-t) \ln \big(\mm(B_{r+\delta}(x_0)) \big)~~\forall t\in [0, 1].
\end{equation}
For any $z\in Z_t$, we have $\d(z, x_0) \leq (1-t)(r+\delta)+\epsilon$.
Choosing  $t=\frac {\delta+\epsilon}{r+\delta} $,   we have $Z_t \subset B_r(x_0)$. Thus  \eqref{prop1:eq1} implies
\[
\ln \big (\mm(B_r (x_0))\big )  \geq  \frac {\delta+\epsilon}{r+\delta}  \ln \big (\mm(B_\epsilon (x_0))\big ) +\frac {r-\epsilon}{r+\delta} \ln \big(\mm(B_{r+\delta}(x_0)) \big).
\]
Then
\begin{eqnarray*}
&&\frac {r-\epsilon}{r+\delta}\left ( \frac{\ln \big(\mm(B_{r+\delta}(x_0)) \big)-\ln \big(\mm(B_{r}(x_0)) \big)}{\delta} \right )\\
&\leq &  \frac {\delta+\epsilon}{\delta(r+\delta)} \Big( \ln \big (\mm(B_r (x_0))\big ) -  \ln \big (\mm(B_\epsilon (x_0))\big ) \Big).
\end{eqnarray*}
From the proof of Theorem \ref{th1},  we can see that $ \frac{\ln \big(\mm(B_{r+\delta}(x_0)) \big)-\ln \big(\mm(B_{r}(x_0)) \big)}{\delta}  \geq C$,  so
\[
\frac {r-\epsilon}{r+\delta} C\leq   \frac {\delta+\epsilon}{\delta(r+\delta)} \Big( \ln \big (\mm(B_r (x_0))\big ) -  \ln \big (\mm(B_\epsilon (x_0))\big ) \Big).
\]
Letting $r \to \infty$, we get
\[
C \leq   \frac {\delta+\epsilon}{\delta} h\ms.
\]
Letting $\epsilon \to 0$ we get the contradiction.
\end{proof}
\bigskip

For metric measure spaces satisfying a generalized Bishop-Gromov inequality, such as $\cdkn$ spaces or $\mcpkn$ spaces with $N<+\infty$, whose volume entropy vanish,   it is known that  there is no isometric inequality in the form of \eqref{th1:eq1} with a positive $C$.
More generally, as a direct consequence of Theorem \ref{th1} and Theorem \ref{th2}, we have the following result. 
\begin{corollary}
Let $\ms$ be a metric measure space with non-negative Ricci curvature  $h\ms=0$, then there is no isoperimetric inequality in the form of 
\[
{\mm^+(\Omega) } \geq C {\mm(\Omega)}~~~\forall \Omega \subset X
\]
from some $C>0$.
\end{corollary}

\bigskip

\begin{example}
Consider the 1-dimensional metric measure space $(\R, | \cdot |, e^{t} \mathcal L^1) $. It can be seen that this is a $\rcd$ space and its volume entropy is 1. Take $\Omega=(-\infty, 0]$, we have $(e^{t} \mathcal L^1)(\Omega)=1$ and $(e^{t} \mathcal L^1)^+(\Omega)=1$, so that the equality holds in \eqref{th1:eq1}.
\end{example}

Inspired by the example above, we propose  the following conjecture, which will be studied in a forthcoming paper.

\begin{conjecture}[Rigidity]\label{th:rigid}
Let $\ms$ be a  $\rcd$ metric measure space with $h\ms>0$. Assume  there is $\Omega \subset X$ with $\mm(\Omega)<+\infty$ such that 
\[
{\mm^+(\Omega) } =h\ms  {\mm(\Omega)}>0.
\]
Then 
 $$\ms \cong \Big (\R, | \cdot |, e^{t} \d t \Big) \times (Y, \d_Y, \mm_Y)$$ for some  $\rcd$ space $(Y, \d_Y, \mm_Y) $ with $\mm_Y(Y)<+\infty$,
and   up to change of variables
\[
\Omega={(-\infty, e] \times Y} \subset  \R \times Y
\]
where $e\in \R$ satisfies  $\int_{-\infty}^e e^s \,\d s=\mm(\Omega)$.
\end{conjecture}

\def\cprime{$'$}

\end{document}